\documentclass{article}[14pt]

\usepackage{amsfonts}
\usepackage{amsmath}
\usepackage{theorem}
\usepackage{array}
\usepackage{t1enc}
\usepackage{color}

\usepackage[T2A,T1]{fontenc} %T2A for Cyrillic symbols
\usepackage[utf8]{inputenc}

\newcommand{\curve}{\chi}
\newcommand{\Fq}{\mathbb{F}_q}

\renewcommand{\vec}[1]{{\bf #1}}

\newcommand{\ev}{\mathrm{ev}}
\newcommand{\im}{\mathrm{im}}
\newcommand{\Pinf}{P_{\infty}}
\newcommand{\val}{v_{\Pinf}}
\newcommand{\ffield}{F}
\newcommand{\MGS}{\mathcal M_{s,\ell}(D,G,A)}

\makeatletter
\def\easycyrsymbol#1{\mathord{\mathchoice
  {\mbox{\fontsize\tf@size\z@\usefont{T2A}{cmr}{m}{n}#1}}
  {\mbox{\fontsize\tf@size\z@\usefont{T2A}{cmr}{m}{n}#1}}
  {\mbox{\fontsize\sf@size\z@\usefont{T2A}{cmr}{m}{n}#1}}
  {\mbox{\fontsize\ssf@size\z@\usefont{T2A}{cmr}{m}{n}#1}}
}}
\makeatother
\newcommand{\Ya}{\easycyrsymbol{\CYRYA}}

\newcommand{\bigO}{\mathcal{O}}
\newcommand{\softO}{\tilde \bigO}

\newtheorem{theorem}{Theorem}

{ \theorembodyfont{\rmfamily}
\newtheorem{example}[theorem]{Example}
}

\newtheorem{lemma}[theorem]{Lemma}
\newtheorem{corollary}[theorem]{Corollary}

\newtheorem{remark}[theorem]{Remark}

\newtheorem{definition}[theorem]{Definition}

\newenvironment{proof}[1][Proof]{\textbf{#1.} }{\ \rule{0.5em}{0.5em}}

\begin{document}

\title{List-decoding of AG codes without genus penalty}
\author{Peter Beelen and Maria Montanucci \thanks{Department of Applied Mathematics and Computer Science, Technical University of Denmark, Kongens Lyngby 2800, Denmark,  {\em emails: } pabe@dtu.dk, marimo@dtu.dk}}
\date{\today}
\maketitle

\begin{abstract}
In this paper we consider algebraic geometry (AG) codes: a class of codes constructed from algebraic codes (equivalently, using function fields) by Goppa. These codes can be list-decoded using the famous Guruswami-Sudan (GS) list-decoder, but the genus $g$ of the used function field gives rise to negative term in the decoding radius, which we call the genus penalty. In this article, we present a GS-like list-decoding algorithm for arbitrary AG codes, which we call the \emph{inseparable GS list-decoder}. Apart from the multiplicity parameter $s$ and designed list size $\ell$, common for the GS list-decoder, we introduce an inseparability exponent $e$. Choosing this exponent to be positive gives rise to a list-decoder for which the genus penalty is reduced with a factor $1/p^e$ compared to the usual GS list-decoder. Here $p$ is the characteristic. Our list-decoder can be executed in $\softO(s\ell^{\omega}\mu^{\omega-1}p^e(n+g))$ field operations, where $n$ is the code length.
\end{abstract}

\begin{small}

{\bf Keywords:} algebraic geometry codes, list-decoding

{\bf 2000 MSC:}  Primary: 94B35, Secondary: 94B27

\end{small}

\section{Introduction}

Algebraic-geometry (AG) codes have received a lot of attention since their introduction by Goppa in \cite{Goppa}, mainly for having provided some of the best error-correcting codes currently known. The fact that AG codes have a good minimum distance is expressed by the Goppa bound on the minimum distance, also known as the designed minimum distance of an AG code. We will denote this designed minimum distance by $d^*$. Many algorithms for decoding AG codes have been introduced in the last three decades.

Skrobogatov and Vladut \cite{SV} first presented an effective decoding method for general AG codes, called the basic algorithm. This algorithm can correct up to $\lfloor (d^*-g-1)/2 \rfloor$ errors, where $g$ is the genus of the curve (or equivalently: function field) from which the code has been constructed. A fast implementation for a large class of AG codes of a decoding algorithm with the same decoding radius is given in \cite{sakata_2018}. The paper \cite{sakata_2018} was the culmination of several papers on decoding AG codes using the Berlekamp-Massey-Sakata algorithm.
Pellikaan \cite{P}, Feng and Rao \cite{FR}, Duursma \cite{D} and Ehrhard \cite{E} have given algorithms for decoding up to the designed minimum distance $d^*$ of the AG code, i.e. for correcting up to $\lfloor (d^*-1)/2 \rfloor$ many errors.  In fact as shown by Kirfel and Pellikaan \cite{KP}, the Feng-Rao algorithm can correct certain AG codes called one-point codes up to $\lfloor (d_{FR}-1)/2 \rfloor$, where $d_{FR} \geq d^*$ is the Feng-Rao bound. The Feng-Rao algorithm is based on Gaussian elimination for solving the system of syndrome equations using majority logic (also known as majority voting), which was generalized by Duursma in \cite{D} for a large class of AG codes. Decoding up to the designed minimum distance can also be realized using the Berlekamp-Massey-Sakata algorithm \cite{sakata_extension_1990}. An overview and many references to the early literature can be found in \cite{handbook}. All these algorithms are bounded distance decoders: the aim is to correct up to a certain number of errors, thus returning a unique codeword or a decoding failure.

Another development started with the celebrated Guruswami-Sudan (GS) list-decoder \cite{GS}. Applied to AG codes, it is capable of decoding beyond the designed minimum distance by returning a list of all codewords within a certain Hamming distance $\tau$ from the received word. Looking back at the bounded distance decoders we discussed, the decoding radius of the basic algorithm from \cite{SV} suffers from a \emph{genus penalty} $g/2$. Likewise, one can show that also the GS list-decoder suffers from a genus penalty, this time of magnitude $g/s$, or actually of magnitude $\ell g /s(\ell+1)$ as we will show. An curious property of the GS decoder when applied to AG codes, is that one can use it to correct up to the designed minimum distance, but only by choosing the multiplicity parameter $s$ and designed list size $\ell$ in the order of magnitude of the genus $g$. This is a somewhat unnatural situation, since one knows in advance that when decoding up to the designed minimum distance, the list size will be only one in case of successful decoding. The aim of this paper is to present a simple variant of the GS list-decoder where the genus penalty is removed. In order to achieve this, we present what we call the \emph{inseparable GS list-decoder with exponent $e$}. Here $e$ can be any nonnegative integer. We will see that for a given $e$, the genus penalty of the resulting list-decoder is reduced with a factor $1/p^e$. If $e=0$, one retrieves the usual GS list-decoder. Choosing $e$ such that $p^e \ge 1+g\ell$, we will be able to remove the genus penalty entirely. This solves the problem of removing the genus penalty in the GS list-decoder completely. Previously, this was only done in case $s=1$ in  \cite{Pan1,Pan2}. Our approach is very different from the approach taken in \cite{Pan1,Pan2}.

To denote algorithmic complexity, we will use the common big-O notation $\bigO(\cdot)$ and its variation, the soft-O notation $\softO(\cdot)$, in which logarithmic factors are omitted. Our complexity estimates will involve the quantity $\omega$, which denotes some real number such that the product of any two matrices in $\mathbb{F}_q^{m \times m}$ can be computed using $\mathcal{O}(m^\omega)$ operations in $\mathbb{F}_q$. The current record with $\omega< 2.37286$ is due to \cite{AW}. An efficient realization of the GS list-decoder, achieving the complexity $\softO(s\ell^{\omega}\mu^{\omega-1}(n+g))$ in the fully general setting of arbitrary AG codes, has been proposed recently in \cite{BRS}. Here $\mu$ denotes the smallest positive element in the Weierstrass semigroup of a chosen rational place, which we will denote by $P_\infty$, of the function field used to construct the AG code. %Continuing the investigation of the decoder from \cite{BRS}, a further complexity improvement was given in \cite{BN}. The list-decoder proposed in \cite{BN} is at least as fast as all existing decoders which are tailored for specific families of AG codes, matching the case of the particularly well studied family of Reed-Solomon (RS) codes.
Using the algorithms developed in \cite{BRS}, we will show that the inseparable GS list-decoder with exponent $e$ can be executed in complexity $\softO(s\ell^{\omega}\mu^{\omega-1}p^e(n+g))$. List-decoding with the genus penalty removed completely from the list-decoding radius can be done in complexity $\softO(s\ell^{\omega}\mu^{\omega-1}(1+g\ell)(n+g))$ if the characteristic is viewed as a constant. No direct comparison with the results for $s=1$ from \cite{Pan1,Pan2} can be made, since there no complexity considerations were specified. However, specializing back to the case of bounded distance decoding ($\ell=s=1$), we will obtain that our algorithm can decode any AG code up to its designed minimum distance in complexity $\softO(\mu^{\omega-1}(1+g)(n+g))$. We will therefore be able to compare this specialization of our algorithm with existing decoders.

We will see that for various good AG codes, such as the class of one-point Hermitian codes, this is faster than the currently fastest known decoder from \cite{sakata_2018}. We will also give an example where the decoder from \cite{sakata_2018} is faster. However, the algorithm from \cite{sakata_2018} does suffer under the genus penalty $g/2$. However, as reported in \cite{sakata_2018}, their decoder does often correct up to $(d^*-1)/2$ many errors, which is why we have chosen to compare our algorithm with theirs anyway. Nonetheless, a comparison with of our decoder with decoding algorithms not suffering a genus penalty is the most fair comparison.  Therefore, we will compare our results with the decoder from \cite{handbook}, which can decode the class of so-called one point codes up to the designed minimum distance $d^*$ (in fact even up to the Feng-Rao bound on the minimum distance). In the best case, the complexity of the decoder presented in \cite{handbook} has the same complexity as the algorithm from \cite{sakata_2018}, but in general the given complexity bound in \cite{handbook} worse. More precisely, the complexity reported in \cite{handbook} is $\bigO(\mu n^2+q^{t+1}(a_1+\cdots+a_t)+tnq^t)$, where $a_1,\dots,a_t$ form a minimal set of generators of the Weierstrass semigroup at $P_\infty$. Another decoder that we will compare our results with is the decoder in \cite{LBO}, which also can decode up to the designed minimum distance and has complexity $\bigO((1+g)(n+g)^2)$. As we will see, our decoder often outperforms the existing decoders. In the case of the well-known one-point Hermitian codes over $\mathbb{F}_{q^2}$ our designed minimum distance decoder has complexity $\softO(q^{\omega+4})$, which outperforms all other currently known algorithms including the one from \cite{sakata_2018}. Finally, note that a decoder which can correct up to the designed minimum distance can be obtained from the usual GS list-decoder by choosing $s,\ell \in \bigO(g)$. However, even using the fast algorithms from \cite{BRS}, this results in a decoder with complexity $\softO(\mu^{\omega-1}(1+g)^{\omega+1}(n+g))$, which is always larger than the complexity of our decoder.

The paper is organized as follows. In Section 2 the necessary setup and background on AG codes and the GS list-decoder are presented. In Section 3 we formulate a variation of the GS list-decoder and analyze its list-decoding radius. In Section 4 we delve into algorithm aspects and derive the complexity of our list-decoder. We finish by studying various examples, comparing the complexity of our decoder with existing ones.

\section{Preliminaries}

\subsection{AG codes}

Let $\mathcal{X}$ be an irreducible, nonsingular, projective algebraic curve $\curve$ of genus $g$ defined over a finite field $\Fq$ of characteristic $p$. We denote the function field of $\mathcal{X}$ by $\ffield$, and by $\mathbb{P}_\ffield$ the set of places of $\ffield$. Any place $P$ of $F$ has a unique discrete valuation ring $\mathcal{O}_P$ associated it in which $P$ is a maximal ideal. The quotient ring $F_P:=\mathcal{O}_P/P$ is a field, called the residue field of $P$. One can show that $F_P$ is a finite extension of $\Fq$. The extension degree $[F_P:\Fq]$ is called the degree of the place $P$. In case $\deg P=1$, the place $P$ is called $\Fq$-rational or simply rational.

A divisor $A$ on $\ffield$ is a formal sum $A=\sum_{P \in \mathbb{P}_\ffield} n_P P$ where $n_P \in \mathbb{Z}$ and $n_P=0$ for all but a finite number of places $P \in \mathbb{P}_\ffield$. The integer $n_P$ is also denoted by $v_P(A)$ and called the valuation of $A$ at the place $P$. The divisor $A$ is called effective, in formulas $A \geq 0$, if $v_P(A) \geq 0$ for all $P \in \mathbb{P}_\ffield$. The degree of $A$ is the sum $\deg(A):=\sum_{P \in \mathbb{P}_\ffield}v_P(A) \deg P$, while the support of $A$ is the finite set $\mathrm{supp}(A):=\{P \in \mathbb{P}_\ffield \mid v_P(A) \ne 0\}$. The Riemann-Roch space $L(A)$ of $A$ is the $\overline{\mathbb{F}}_q$-vector space given by

$$L(A):=\{f \in \ffield \setminus \{0\} \mid (f) + A \geq 0\} \cup \{0\}.$$

We denote by $\ell(A):=\dim L(A)$ its dimension over the constant field $\Fq$. The Riemann-Roch theorem \cite[Theorem 1.5.15]{S} implies that $\ell(A) \geq \deg(A)+1-g$ and that equality holds if $\deg(A) \geq 2g-1$. Moreover $\ell(A)=0$ if $\deg(A)<0$ as the degree of a principal divisor is zero. For more information on function fields, see \cite{S}.

\begin{definition} Assume that $\ffield$ has at least $n$ rational places. For $P_1,\dots,P_n$ pairwise distinct rational places of $\ffield$ define $D$ to be the divisor $D=P_1+\cdots+P_n$. Let $G$ be a divisor whose support is disjoint from that of $D$.
Then we define the algebraic geometry (AG) code: \[C_L(D,G)=\{ev_D(f)=(f(P_1),\dots,f(P_n)) \mid f \in L(G)\} \subset \mathbb{F}_q^n.\]
\end{definition}

For further details on this code, see \cite[Chapter 2]{S}. Here we state some properties of $C_L(D,G)$ that we will use in the following.
The code $C_L(D,G)$ has designed minimum distance $d^*=n-\deg(G)$. We assume that $\deg G \le n+2g-1$, since for $\deg G \ge n+2g-1$, we have $C_L(D,G)=\Fq^n.$ In fact, we will later see that we can even make a stronger assumption on the degree of $G$, see the discussion after Theorem \ref{thm:GS_new}.

\begin{remark}
In his original construction Goppa considered AG codes $C_\Omega(D,G)$ defined by using residues of certain differentials. These codes can also be obtained as evaluation codes \cite[Proposition 2.2.10]{S}. Hence a decoder that can handle the AG codes defined above, can also handle codes of the form $C_\Omega(D,G)$.
\end{remark}

\subsection{Guruswami-Sudan Decoding}\label{subsec:GS}

The aim of this section is to describe the Guruswami-Sudan list-decoding \cite{GS} for $C_L(D, G)$.

%as presented in \cite{BRS}, that is, in terms of $\Ya$ modules.

To do so, we fix $s, \ell \in \mathbb{N}$, with $s \leq \ell$ (called multiplicity parameter and designed list size, respectively). The corresponding list-decoding radius will be denotes by $\tau$.

\begin{definition}
Let $P$ be a rational place of $\ffield$, $r \in \Fq$ and $Q(z) \in \ffield[z]$. One says that $Q$ has a root of multiplicity $s$ at $(P,r)$ if for a local parameter $\phi$ of $P$, there exist $c_ {a,b} \in \Fq$ such that
$$Q(z)=\sum_{a,b \geq 0, \ a+b \geq s} c_{a,b}\phi^a(z-r)^b, $$
with $c_{a,s-a} \ne 0$ for at least one value $0 \leq a \leq s$.
\end{definition}

An immediate consequence of the triangle inequality is that if  $Q \in \ffield[z]$ has a root of multiplicity $s$ at $(P,r)$, and $f \in \ffield$ is
such that $f(P)=r$, then $v_P(Q(f)) \geq s$, see \cite[Lemma 4.2]{BRS}.

The list-decoding using the Guruswami-Sudan algorithm is based on the following well-known result:

\begin{theorem}\label{thm:GS}
Let $f \in L(G)$ and write $\vec c=(f(P_1),\dots,f(P_n))$ for the corresponding codeword in $C_L(D,G)$. Further, let $\vec r \in \Fq^n$, $\vec e=\vec r-\vec c$ the error vector and write $\tau=w_H(\vec e).$  Further, let $s \le \ell$ be positive integers. Now let $A$ be any divisor whose support is disjoint from $D$ such that $\deg A < s(n-\tau)$. Then, any nonzero $Q(z) \in \ffield[z]$ of the form $Q(z)=Q_\ell z^{\ell}+\cdots + Q_0$, such that
\begin{enumerate}
\item $Q_{j} \in L(A-jG)$ for $j=1,\dots,\ell$ and
\item the points $(P_i,r_i)$ are zeroes of $Q(z)$ with multiplicity at least $s$ for $i=1,\dots,n$,
\end{enumerate}
satisfies $Q(f)=0$.
\end{theorem}

The existence of a nonzero $Q(z)$ satisfying the above conditions is usually guaranteed by requiring that the number of coefficients in $Q$, that is $\sum_{i=0}^\ell \dim L(A-iG)$, is strictly larger than $\binom{s+1}{2}n$, which is the number of linear equations these coefficients need to satisfy in order for the points $(P_i,r_i)$ to be zeroes of $Q(z)$ with multiplicity at least $s$. Using the estimate $\dim L(A-jG) \ge \deg(A-jG)-g+1$, this implies that a divisor $A$ and a nonzero $Q$ as in Theorem \ref{thm:GS} exist whenever
\begin{equation}\label{eq:GSdec_radius}
\tau \le \frac{s(2\ell-s+1)n-\ell(\ell+1)\deg G -2}{2s(\ell+1)} -\frac{g}{s}.
\end{equation}
We will call the term $g/s$ the \emph{genus penalty}. The goal is to find efficient algorithms that have a smaller or no genus penalty. Our first and main contribution in this paper is to show that a simple trick involving the Freshman's dream can be used to generalize Theorem \ref{thm:GS}, achieving a better decoding radius.

\section{Improving the list-decoding radius of Guruswami-Sudan using the Freshman's dream}\label{sec:Freshman}

As mentioned before, the aim of this section is to show how a simple trick involving the Freshman's dream can be used to generalize Theorem \ref{thm:GS}, achieving a better decoding radius. The main theorem also involves a parameter $t$, but this parameter plays no role in the actual decoding algorithm. It is solely used to analyze the decoding radius: for each choice of $t$ between $0$ and $s$, an upper bound on the decoding radius is found. Finding a general expression for $t$ that maximizes this upper bound seems a nontrivial task, but $t=1$ seems to give the main improvement. Therefore we state that case as an important corollary.

\begin{theorem}\label{thm:GS_new}
Let $f \in L(G)$ and write $\vec c=(f(P_1),\dots,f(P_n))$ for the corresponding codeword in $C_L(D,G)$. Further, let $\vec r \in \Fq^n$, $\vec e=\vec r-\vec c$ the error vector and write $\tau=w_H(\vec e).$  Further, let $0 \le t \le s \le \ell$ be integers and assume that $s \ge 1$. Let us write \[\tau_t^{(e)}:=\frac{n(s(\ell-t+1)-\binom{s-t+1}{2})-(\binom{\ell+1}{2}-\binom{t}{2})\deg G}{\binom{s+1}{2}+s(\ell-t+1)-\binom{s-t+1}{2}}-\frac{1+g(\ell-t+1)}{p^e\left(\binom{s+1}{2}+s(\ell-t+1)-\binom{s-t+1}{2}\right)},\] where $p$ is the characteristic and $e \ge 0$ an integer. Assume that $\tau \le \tau_t^{(e)}$. Then for any divisor $A$ with support disjoint from $D$ satisfying $\deg A < p^es(n-\tau)$, there exists a nonzero $Q(z) \in \ffield[z]$ of the form $Q(z)=Q_\ell z^{p^e\ell}+\cdots +Q_1 z^{p^e}+ Q_0$, such that
\begin{enumerate}
\item $Q_{j} \in L(A-p^ejG)$ for $j=1,\dots,\ell$, and
\item the points $(P_i,r_i)$ are zeroes of $Q(z)$ with multiplicity at least $p^es$ for $i=1,\dots,n$.
\end{enumerate}
Any such polynomial $Q(z)$ has at most $\ell$ roots, and $f$ is one of these.
\end{theorem}
\begin{proof}
Since divisors with support disjoint to $D$ of any degree exist, the existence of the divisor $A$ is trivial. Now let $A$ be chosen arbitrarily. We will show that there exists a nonzero polynomial of the form $Q(z)=(z-f)^{p^et}\tilde{Q}(z)$, with $\tilde{Q}(z)=\tilde{Q}_{p^e(\ell-t)}z^{\ell-t}+\cdots+\tilde{Q}_1z^{p^e}+\tilde{Q}_0$, satisfying all requirements. First of all, to make sure that the coefficient $Q_j$ of $z^{p^ej}$ in $Q(z)$ is in $L(A-p^ejG)$ for all $j$, we simply require that $\tilde{Q}_j \in L(A-p^e(j+t)G)$. Choosing bases of the spaces $L(A-p^e(j+t)G)$, we can express $\tilde{Q}(z)$ as an element of an $\mathbb{F}_q$-vector space of dimension $\sum_{j=0}^{\ell-t}\dim L(A-p^e(j+t)G).$

Now denote by $\mathcal E=\{P_{i_1},\dots,P_{i_\tau}\}$ the set of error positions. Since for $P_i \not\in \mathcal E$, the point $(P,r_i)$ equals $(P_i,f(P_i))$, the polynomial $(z-f)^{p^et}$ has $(P_i,f(P_i))$ as a zero of multiplicity $p^et$. This means that in order to make sure that the points $(P_i,r_i)$ are zeroes of $Q(z)$ with multiplicity at least $p^es$ for $i=1,\dots,n$, we need that $(P_i,r_i)$ is a zero of $\tilde{Q}(z)$ with multiplicity at least $p^e(s-t)$ if $P \not\in \mathcal E$, and a zero of $\tilde{Q}(z)$ with multiplicity at least $p^es$ if $P \in \mathcal E.$ Now let $t_i$ be a local parameter of $P_i$. Since the exponents of $z$ occurring in $\tilde{Q}(z)$ all are a multiple of $p^e$, the Freshman's dream implies that $\tilde{Q}(z)$ can be written in the form $\sum_{a \ge 0} \sum_{b=0}^\ell c_{a,b}^{(i)} t_i^a (z-r_i)^{p^eb},$ for certain $c_{a,b}^{(i)} \in \mathbb{F}_q$ Hence $(P_i,r_i)$ is a zero of $\tilde{Q}(z)$ of multiplicity at least $s$ if and only if $c_{a,b}^{(i)}=0$ for $b=0,\dots,s-1$ and $a=0,\dots,p^e(s-b)-1.$ In particular, having $(P_i,r_i)$ as a zero of multiplicity at least $s$ can be expressed using $p^e \binom{s+1}{2}$ linear equations. Similarly having $(P_i,r_i)$ as a zero of multiplicity at least $s-t$ can be expressed using $p^e \binom{s-t+1}{2}$ linear equations. This means the total number of linear equations for $\tilde{Q}(z)$ equals \[p^e\tau \binom{s+1}{2}+p^e(n-\tau)\binom{s-t+1}{2}.\]

To guarantee the existence of a nonzero $\tilde{Q}(z)$, we simply need that
\[\sum_{j=0}^{\ell-t}\dim L(A-p^e(j+t)G)  > p^e\tau \binom{s+1}{2}+p^e(n-\tau)\binom{s-t+1}{2}.\] Using that $\dim L(A-p^e(j+t)G) \ge \deg(A-p^e(j+t)G)-g+1$, we see that it is sufficient to require
\[\sum_{j=0}^{\ell-t} \deg(A)-g+1-p^e(j+t)\deg G \ge p^e\tau \binom{s+1}{2}+p^e(n-\tau)\binom{s-t+1}{2}+1.\]
Combining this with the requirement that $\deg A \le p^es(n-\tau)-1$, we obtain that a nonzero $\tilde{Q}(z)$, and hence $Q(z)$, exists as long as $\tau \le \tau_t^{(e)}.$

Since $Q(f) \in L(A-\sum_{i: P_i \not\in \mathcal E}p^esP_i)$ and $\deg A< p^es(n-\tau)$, we see $Q(f)=0.$ Moreover, any root $g \in \ffield$ of $Q(z)$ has multiplicity at least $p^e$, since, using the Freshman's dream, it is easy to see that $Q(z+g)$ is divisible by $z^{p^e}$. This implies that $Q(z)$ has at most $\ell$ roots.
\end{proof}

Note that if the designed list size $\ell$ is chosen such that $\deg(A-p^e\ell G) <0$, then $L(A-\ell G)=\{0\}$, meaning that $Q_\ell=0$ for any $Q(z)$ satisfying the conditions in Theorem \ref{thm:GS_new}. To avoid that the leading term of $Q(z)$ cannot be nonzero, we therefore require that the parameters $e$, $s$ and $\ell$ are chosen such that $\deg (A -p^e\ell G) \ge 0$. Since $\deg A < p^esn$, this implies that we may assume that:
\begin{equation}\label{eq:G_nslinv}
\deg G \le s \ell^{-1} n.
\end{equation}
The fact that the polynomial $Q(z)$ can have at most $\ell$ roots, even though its $z$-degree is $p^e \ell$, means that it is still meaningful to call $\ell$ the designed list size, as is usual in the context of the Guruswami-Sudan list-decoder. We call the list-decoder Theorem \ref{thm:GS_new} gives rise to, the \emph{inseparable Guruswami-Sudan list-decoder} (with exponent $e$). The reason for the term ``inseparable'' is that for $e>0$, the polynomials $Q(z)$ occurring in Theorem \ref{thm:GS_new} are inseparable polynomials.

For $e=0$, one obtains the usual Guruswami-Sudan list-decoder as described in Subsectin \ref{subsec:GS}. Also note that for $t=e=0$, one simply recovers the usual list-decoding radius from equation \eqref{eq:GSdec_radius}. There are some other cases, we would like to highlight as corollaries. As stated previously, the most important result is obtained by setting $t=1$.

\begin{corollary}\label{cor:insepGS_t1}
The inseparable Guruswami-Sudan list-decoder with exponent $e$ has list-decoding radius at least
\begin{equation*}%\label{eq:GSdec_radius}
\frac{s(2\ell-s+1)n-\ell(\ell+1)\deg G}{2s(\ell+1)}-\frac{1+g\ell}{p^e s(\ell+1)}.
\end{equation*}
In particular, choosing $e=\lceil \log_p\left(1+g\ell\right)\rceil$, one obtains a list-decoder with list-decoding radius at least
\[\frac{s(2\ell-s+1)n-\ell(\ell+1)\deg G-2}{2s(\ell+1)}.\]
\end{corollary}
\begin{proof}
The first part follows directly from Theorem \ref{thm:GS_new} by putting $t=1$. The second part follows, since if $e=\lceil \log_p\left(1+g\ell\right)\rceil$, then \[\frac{1+g\ell}{p^e s(\ell+1)} \le \frac{1}{s(\ell+1)} = \frac{2}{2s(\ell+1)}.\]
\end{proof}

Further putting $e=0$, we obtain:

\begin{corollary}\label{cor:GS_decodingrad}
The usual Guruswami-Sudan list-decoder has list-decoding radius at least
\begin{equation*}%\label{eq:GSdec_radius}
\tau \le \frac{s(2\ell-s+1)n-\ell(\ell+1)\deg G -2}{2s(\ell+1)} -\frac{\ell}{\ell+1}\cdot\frac{g}{s}.
\end{equation*}
\end{corollary}

The list-decoding radius in Corollary \ref{cor:GS_decodingrad} is quite similar to the one mentioned in equation \eqref{eq:GSdec_radius}. The only difference is that the genus penalty is reduced from $g/s$ to $\ell g/s(\ell+1)$. This reduction is not due to any change of algorithms, but because of a slightly more refined analysis. This in particular shows that if $s=\ell=1$ and $e=0$, then the genus penalty is reduced from $g$ to $g/2$. That this is possible, may not come as a surprise, since the decoding algorithm from \cite{SV}, commonly known as the basic algorithm, already has the same decoding radius $(d^*-1-g)/2$. Indeed, the two decoders are closely related. Setting $s=\ell=1$, but keeping $e$ free. In this case one also obtains a bounded distance decoder, with an improved decoding radius compared to the basic algorithm.

\begin{corollary}\label{cor:q_mindistdec}
The inseparable Guruswami-Sudan list-decoder with $s=\ell=1$ and exponent $e$, is a bounded distance decoder for the AG-code $C_L(D,G)$ with decoding radius at least $\frac{d^*}{2}-\frac{1+g}{2p^e}$, where $d^*=n-\deg G$ is the designed minimum distance of $C_L(D,G)$.
In particular, choosing $e=\lceil \log_p\left(1+g\right)\rceil$, one obtains a decoder with decoding radius $(d^*-1)/2$.
\end{corollary}
\begin{proof}
The first part follows by setting $\ell=s=t=1$ in Theorem \ref{thm:GS_new}. The second part follows, since for the given choice of $e$, one has $\frac{1+g}{2p^e} \le \frac12$.
\end{proof}

We will see later what the precise effect of $e$ is on the complexity of the inseparable Guruswami-Sudan list-decoder. As it turns out, the effect is that the complexity is multiplied with the factor $p^e$ (ignoring logarithmic factors). Hence the second part of Corollary \ref{cor:q_mindistdec} can be interpreted as follows: removing the genus penalty $g/2$ can be done by increasing the complexity by a factor $p^{\lceil \log_p\left(1+g\right)\rceil} \approx 1+g$. In \cite{P} and \cite{Vlad1990}, the observation was made that one theoretically can decode up to $\frac{d^*-1}{2}$ many errors by running $2g$ suitably chosen instances of the basic algorithm from \cite{SV}. The element of choice lies in choosing $2g$ divisors $A$ to be used in the basic algorithms. Similarly, we need to choose a divisor $A$ in Theorem \ref{thm:GS_new}. However, the results in \cite{P} and \cite{Vlad1990} only indicate the existence of these $2g$ divisors, but do not explain how to find these divisors explicitly. In that sense Corollary \ref{cor:q_mindistdec} is much easier to use: any choice of the divisor $A$ is fine, as long as it satisfies the requirements from Theorem \ref{thm:GS_new}.

Going back to list-decoding, if in Corollary \ref{cor:insepGS_t1} we choose $e=\lceil \log_p\left(1+g \ell\right)\rceil$, then the decoding radius of the inseparable Guruswami-Sudan list-decoder becomes exactly the same as in equation \eqref{eq:GSdec_radius} for genus zero curves. Hence the genus penalty in the Guruswami-Sudan list-decoder can be removed at the cost of increasing the complexity with a factor $p^e \approx 1+g\ell$. In case $s=1$, the Guruswami-Sudan list-decoder is often called the Sudan list-decoder. In Chapter 3 of \cite{Pan1} (also see \cite{Pan2}), the existence of a Sudan-like list-decoding algorithm was given that removes the genus penalty. Their approach follow the ideas from \cite{P} and \cite{E}. No complexity for their algorithm is stated, so no direct comparison in terms of complexity is available. On the theoretical side, it is worth noting that while in \cite{Pan1} one assumes $s=1$, our approach works for any $s \ge 1$.

\section{Algorithmic aspects of the inseparable GS list-decoder}

In this section, we consider algorithmic aspects of the decoders presented in the previous two sections. We will use the same approach from \cite{BRS}. In fact, at an essential point, we will be able to use algorithms described in \cite{BRS} to determine the complexities of our list-decoder.

As in \cite{BRS}, we assume that there exists a rational place $\Pinf$ not used as evaluation point. At first sight, this seems to restrict the AG codes one can decode, but, as shown in \cite{BRS}, actually no loss of generality occurs by assuming $\Pinf$ exists. The reason is that one can always extend the field $\Fq$ of definition in a modest way so that \emph{after} the extension, the function field contains places of degree one not used as evaluation points. More precisely, as mentioned before, the needed extension degree is logarithmic in $n$ and $g$. Therefore the overall complexity of the decoding algorithms is not affected significantly by this field extension. Actually, in the $\softO$ notation, no effect can be observed at all. As in \cite{BRS}, the complexity-expressions will involve several parameters from the (list)-decoding setup and the used code and also the quantity $\mu$. Here, as in \cite{BRS}, $\mu$ denotes the smallest positive element from the Weierstrass semigroup of $P_\infty$.

As explained in \cite{BRS}, one may assume without loss of generality that the divisor $G$ is effective. %Indeed, as mentioned in \cite{BRS}, if $G$ is not effective, then possibly after puncturing and passing to a slightly larger field of definition, an AG code equivalent to $C_L(D,G)$ can be found of the form $C_(D,G')$, where $G'$ is effective and has support disjoint from $D$.
We will make the same assumption about $\Pinf$ and $G$ from now on. Recall, that we also always work under the assumption that $\deg G \le sn/\ell \le n$, see equation \eqref{eq:G_nslinv}.

The ring $\Ya=\cup_{m \ge 0} L(mP_\infty)$ is an integral domain, being a subset of the function field $\ffield$ of $\curve$. In fact, it is even a Dedekind domain, though not necessarily a principal ideal domain. Given any $\Fq$-rational divisor $A$ on $\curve$, we define $\Ya(A)=\cup_{m \ge 0} L(A+mP_\infty)$. It is a torsion free, rank one module over $\Ya$. However, since $\Ya$ may not be a principal ideal domain, $\Ya(A)$ need not be a free module. It is always possible to generate $\Ya(A)$ by two suitably chosen elements. To indicate the size of there generators, we define for any nonzero $a \in \Ya(A)$, the quantity $\delta_A (a)$ to be the smallest integer $m$ such that $a \in L(m\Pinf + A)$, i.e., $\delta_A(a) = -\val(a) - \val(A)$ and let $\delta(a) = \delta_0(a) = -\val(a)$. We will take as convention that $\delta_A(0) = - \infty$. Note that for any $a \in \Ya(A)$ and $b \in \Ya(B)$, one has $\delta_{A+B}(ab)=\delta_A(a)+\delta_B(b).$ Moreover, if $A=mP_\infty$ for some integer $m$ and $B$ is any divisor, then $\Ya(A+B)=\Ya(B)$. The following lemma is a direct consequence of \cite[Lemma III.1]{BRS}.

\begin{lemma}[\cite{BRS}]\label{lem:YaA-over-Ya}
Let $A=\sum_{i=1}^t n_iQ_i$ be a divisor of $\ffield$ and write $\mathfrak a(A)= \sum_{i} \deg Q_i$.
Then $\Ya(A)$ can be generated as a $\Ya$-module by two elements $a_1$ and $a_2$ satisfying $\delta_A(a_u) \le 4g-1-\deg(A)+\mathfrak a(A)$ for $u=1,2$.
\end{lemma}

Note that $\mathfrak a (A+B) \le \mathfrak a(A)+\mathfrak a(B)$ for any two divisors $A$ and $B$, and that equality holds if the supports of $A$ and $B$ are disjoint. In particular $\mathfrak a (A+B)$ can be large even if $\deg (A+B)$ is small. However, for an effective divisor $A$, we have $\mathfrak a(A) \le \deg A$ and $\mathfrak a(-A) \le \deg A$.

Given divisors $D=P_1+\dots+P_n$ and $G$ used to define the code $C_L(D,G)$ and a divisor $A$ used for list-decoding as in Theorem \ref{thm:GS_new}, we define
\[\MGS^{(e)} = \left\{\sum_{j=0}^\ell Q_j z^{p^ej} \in \ffield[z] \mid Q_j \in \Ya(A-p^e jG); \ Q(P_i,r_i)=0 \text{ with multiplicity at least $p^es$}\right\}.\]
Note if additionally, one knows that $Q_j \in L(A-p^e tG)$, one obtains the polynomials discussed in Theorem \ref{thm:GS_new}. For any $Q(z) = \sum_{j = 0}^\ell Q_j z^{p^ej}$ with $Q_j \in \Ya(A-p^ejG)$ we define $\delta_{G,A}^{(e)}(Q(z)) = \max_j \delta_{A-p^ejG}(Q_j)$. Note that $\delta_{G,A}^{(e)}(Q(z)) \le 0$, precisely if $\delta_{A-p^ejG}(Q_j) \le 0$ for all $j$, or equivalently, precisely if $Q_j \in L(A-p^ejG)$ for all $j$. In particular, the polynomials satisfying the requirements in Theorem \ref{thm:GS_new} are precisely those polynomials $Q(z) \in \MGS^{(e)}$ satisfying $\delta_{G,A}^{(e)}(Q) \le 0.$

The following theorem describes the structure of $\MGS^{(e)}$ as a $\Ya$-module. For $e=0$ and $A$ a multiple of $P_\infty$, one obtains Theorem IV.4 from \cite{BRS}. The proof of the following theorem is a direct modification of the proof of Theorem IV.4 as given in \cite{BRS}, but for the sake of completeness, we include it. In this subsection, we could restrict to the case that $A$ is a multiple of $P_\infty$, but further on in the paper, the divisor $A$ will not always have that form.

\begin{theorem}
  \label{thm:M-description}
Let $\vec{r}=(r_1,\dots,r_n)$ be a received word and $R \in \Ya(G)$ be such that $R(P_i) = r_i$ for $i = 1,\dots,n$. Then it holds that
  \[
    \MGS^{(e)} = \bigoplus_{j=0}^{s}(z-R)^{p^ej} \Ya(A+p^eG_j) \oplus \bigoplus_{j=s+1}^{\ell}(z-R)^{p^es}z^{p^e(j-s)} \Ya(A+p^eG_j),
  \]
where $G_j = -jG - \max\{0,s-j\}D$ for $j = 0,\dots,\ell$.
\end{theorem}
\begin{proof}
Note that for all $i$ and all $h \in \Ya(p^eG_j)$, $v_{P_i}(h) \ge p^e\max\{0,s-j\}$. Further $(z-R)^{p^e j}$ has a root of multiplicity $p^e j$ at $(P_i,r_i)$, since $R(P_i)=r_i$. Hence any element in $(z-R)^{jp^e} \Ya(A+p^e G_j)$ has a root of multiplicity at least $p^e s$ at $(P_i,r_i)$.
Moreover, using the Freshman's dream and the fact that $R \in \Ya(G)$, we see that $(z-R)^{p^ej} \Ya(A+p^eG_j)=\left(\sum_{u=0}^j z^{p^e u}\binom{j}{u}(-R)^{p^e(j-u)}\right)\Ya(A+p^e G_j) \subseteq \bigoplus_{u=0}^jz^{p^e u} \Ya(A+p^e G_u)$. Hence $(z-R)^{p^e j} \Ya(A+p^e G_j) \subseteq \MGS^{(e)}$ for $j=0,\dots,s$, while for $j=s+1,\dots,\ell$ one similarly obtains $(z-R)^{p^e s} z^{p^e(j-s)} \Ya(A+p^e G_j) \subseteq \MGS^{(e)}$. All in all, we may conclude that
\[\bigoplus_{j=0}^{s}(z-R)^{p^e j} \Ya(A+p^e G_j) \oplus \bigoplus_{j=s+1}^{\ell}(z-R)^{p^e s}z^{p^e(j-s)} \Ya(A+p^e G_j) \subseteq  \MGS^{(e)}.\]

We will now prove the reverse inclusion
\begin{equation}\label{eq:inclusion}
\MGS^{(e)} \subseteq  \bigoplus_{j=0}^{s}(z-R)^{p^e j} \Ya(A+p^e G_j) \oplus \bigoplus_{j=s+1}^{\ell}(z-R)^{p^e s}z^{p^e(j-s)} \Ya(A+p^e G_j).
\end{equation}
First of all, let us assume that the theorem is true for $e=0$. Note $\MGS^{(e)} \subseteq \mathcal{M}_{p^e s, p^e \ell}(D,G,A)^{(0)}$. Then, using the assumption that the theorem is true for $e=0$,
\[\mathcal{M}_{p^e s, p^e \ell}(D,G,A)^{(0)} =\bigoplus_{a=0}^{p^e s}(z-R)^{a} \Ya(A+H_a) \oplus \bigoplus_{a=p^es+1}^{p^e \ell}(z-R)^{p^es}z^{a-p^es} \Ya(A+H_a),\] with $H_a = -aG - \max\{0,p^es-a\}D$ for $a = 0,\dots,p^e\ell$. Since by definition, the $z$-degrees occurring in an element of $\MGS^{(e)}$ are all multiples of $p^e$, we obtain that
\[\MGS^{(e)} \subseteq \bigoplus_{j=0}^{s}(z-R)^{p^e j} \Ya(A+H_{p^e j}) \oplus \bigoplus_{j=s+1}^{\ell}(z-R)^{p^e s} z^{p^e(j-s)} \Ya(A+H_{p^e j}).\] Using that $H_{p^e j}=p^e G_j$, the inclusion from equation \eqref{eq:inclusion} follows.

What remains to be shown is the inclusion
\begin{equation}\label{eq:inclusion0}
\MGS^{(0)} \subseteq  \bigoplus_{j=0}^{s}(z-R)^{j} \Ya(A+G_j) \oplus \bigoplus_{j=s+1}^{\ell}(z-R)^{s}z^{j-s} \Ya(A+G_j).
\end{equation}
We will prove this using induction on $s$.
Let $Q(z)=\sum_{j = 0}^{\ell}Q_jz^j \in \MGS^{(0)}$ and write $Q(z)=\sum_{j = 0}^{\ell}\tilde{Q}_j(z-R)^j$ for certain $\tilde{Q}_j \in F$. Writing $z^j=((z-R)+R)^j$ and using Newton's binomium, we obtain
\[\tilde{Q}_j=\sum_{u=j}^\ell \binom{u}{j}R^{u-j} Q_u \in \Ya(A-jG), \text{ for $j=0,\dots,\ell$,}\]
since $R \in \Ya(G).$
Hence $Q(z)=\sum_{j = 0}^{s}\overline{Q}_j(z-R)^j + \sum_{j = s+1}^{\ell}\overline{Q}_j(z-R)^s z^{j-s}$, with
\[\overline{Q}_j=\tilde{Q}_j \text{ for $j=0,\dots,s-1$ and } \overline{Q}_j=\sum_{u=j}^{\ell}\tilde{Q}_u \binom{u-s}{j-s}(-R)^{u-j} \text{ for $j=s,\dots,\ell$.}\]
Note that $\overline{Q}_j \in \Ya(A-jG)$ for $j=0,\dots,\ell$. Also observe that $\overline{Q}_0=\tilde{Q}_0=Q(R) \in \Ya(A+G_0),$ since $Q(z)$ has $(P_i,r_i)$ as a zero of multiplicity at least $s$ and $R(P_i)=r_i$.

Now if we assume $s=1$, then $\Ya(A+G_j)=\Ya(A-jG)$ for $j>0$ and we can conclude from the above that $Q(z) \in \bigoplus_{j=0}^{s}(z-R)^j \Ya(A+G_j) \oplus \bigoplus_{j=s+1}^{\ell}(z-R)^sz^{j-s} \Ya(A+G_j)$.

If $s>1$, we proceed as follows: using $\overline{Q}_0 \in \Ya(A+G_0) \subseteq \MGS^{(0)}$, we conclude
\[\MGS^{(0)} \ni Q(z)-\overline{Q}_0=(z-R) \cdot \left(\sum_{j = 0}^{s-1}\overline{Q}_{j+1}(z-R)^j + \sum_{j = s}^{\ell-1}\overline{Q}_{j+1}(z-R)^{s-1}z^{j-s+1}\right) \, .\]
Since $z-R$ has a root of multiplicity one at $(P_i,r_i)$ for all $i$, we see that $$\sum_{j = 0}^{s-1}\overline{Q}_{j+1}(z-R)^j + \sum_{j = s}^{\ell-1}\overline{Q}_{j+1}(z-R)^{s-1}z^{j-s+1}$$ has a root of multiplicity at least $s-1$ at $(P_i,r_i)$ for all $i$. Hence $$\sum_{j = 0}^{s-1}\overline{Q}_{j+1}(z-R)^j + \sum_{j = s}^{\ell-1}\overline{Q}_{j+1}(z-R)^{s-1}z^{j-s+1} \in \mathcal M_{s-1,\ell-1}(D,G,A)^{(0)}.$$ Then using the induction hypothesis for $s-1$, we may conclude that  $Q(z) \in \bigoplus_{j=0}^{s}(z-R)^{j} \Ya(A+G_j) \oplus \bigoplus_{j=s+1}^{\ell}(z-R)^{s}z^{j-s} \Ya(A+G_j)$ so that the inclusion in equation \eqref{eq:inclusion0} follows.
\end{proof}

\begin{corollary}\label{cor:sizegenM}
For $j=0,\dots,\ell$, choose $\{g_1^{(j)},g_2^{(j)}\}$ a generating set of the $\Ya$-module $\Ya(A+p^eG_j)$ satisfying
\[\delta_{A+p^eG_j}(g_u^{(j)}) \le 4g-1+(p^ej+1)\deg G+p^e\max\{0,s-j+1\}n-\deg A+\mathfrak a(A)\]
for $u=1,2$. Further, choose $R \in \Ya(G)$ such that $\delta_G(R) \le n+2g-1-\deg G$.
Then
\[\{(z-R)^{p^ej} g_u^{(j)} \mid j=0,\dots,s; \ u=1,2\} \cup \{(z-R)^{p^es}z^{p^e(j-s)} g_u^{(j)} \mid j=s+1,\dots,\ell; \ u=1,2\}\]
is a generating set of the $\Ya$-module $\MGS^{(e)}$ and moreover for any $j=0,\dots,s$ and $u=1,2$,
\[\delta_{G,A}((z-R)^{p^ej} g_u^{(j)}) \le (p^e(j+s)+2)n+(p^ej+3)(2g-1)+1 - \deg A +\mathfrak a(A),\]
while for any $j=s+1,\dots,\ell$ and $u=1,2$,
\[\delta_{G,A}((z-R)^{p^es}z^{p^e(j-s)} g_u^{(j)}) \le 2p^esn+(p^es+2)(2g-1)+1 - \deg A +\mathfrak a(A).\]
\end{corollary}
\begin{proof}
First of all, Lemma \ref{lem:YaA-over-Ya} implies the existence of the indicated $g_1^{(j)}$ and $g_2^{(j)}$. The evaluation map $\ev_D: L(G+(n+2g-1-\deg G)\Pinf) \to \Fq^n$ defined by $\ev_D(f)=(f(P_1),\dots,f(P_n))$ is surjective, since
$$\dim \im(\ev_D) = \dim L(G+(n+2g-1-\deg G)\Pinf) - \dim L(G+(n+2g-1-\deg G)\Pinf-D) = (n+g)-(g)=n.$$
Therefore the indicated $R$ also exists. The fact that the indicated set %$\{(z-R)^{p^ej} g_u^{(j)} \mid j=0,\dots,\ell; \ u=1,2\}$
is a generating set of the $\Ya$-module $\MGS^{(e)}$ follows from Theorem \ref{thm:M-description}.

The last parts follows by a direct calculation from the previous. Indeed, for $j=0,\dots,s$ and $u=1,2$, the coefficient of $z^{p^et}$ in $(z-R)^{p^ej} g_u^{(j)}$  equals $\binom{j}{t}(-R)^{p^e(j-t)}g_u^{(j)}.$
This implies that \[\delta_{A-p^etG}(\binom{j}{t}(-R)^{p^e(j-t)}g_u^{(j)}) \le (p^e(j-t+s)+1)n+(p^e(j-t)+2)(2g-1)+1+(p^et+1)\deg G - \deg A +\mathfrak a(A).\]
Since $\deg G \le n+2g-1$, we see that the maximum is attained for $t=0$, whence
\[\delta_{G,A}((z-R)^{p^ej} g_u^{(j)}) \le (p^e(j+s)+1)n+(p^ej+2)(2g-1)+1+\deg G - \deg A +\mathfrak a(A).\] This implies the stated upper bound for $\delta_{G,A}((z-R)^{p^ej} g_u^{(j)}).$

For $j=s+1,\dots,\ell$ and $u=1,2$, the coefficient of $z^{p^e(t+j-s)}$ in $(z-R)^{p^es}z^{p^e(j-s)} g_u^{(j)}$  equals $\binom{s}{t}(-R)^{p^e(s-t)}g_u^{(j)}.$ From a similar computation as above, but this time using that $G_j=-jG$ and $\deg G \le s \ell^{-1} n$, see equation \eqref{eq:G_nslinv}, the result follows directly.
\end{proof}

At this point, the decoding algorithm described in \cite{BRS} can nearly be applied verbatim, especially if we set $A=0$. In fact for $e=0$ and $A=0$, we are precisely in the setting of \cite{BRS}, where it was shown the resulting list-decoder can be performed in $\softO(s \ell^\omega \mu^{\omega-1}(n+g))$ operations. The main approach in \cite{BRS}, following a similar point of view in \cite{LO1,LO2,BB}, is to view the $\Ya$-module $\mathcal M_{s,\ell}(D,G,0)^{(0)}$ as a module over $\Fq[x]$, where $x \in F$ is a suitably chosen element having a pole at $P_\infty$. The rank of $\mathcal M_{s,\ell}(D,G,0)^{(0)}$ as a $\Ya$-module is $\ell+1$, but as an $\Fq[x]$-module, the rank becomes $\mu(\ell+1)$. Using the generators found in Corollary \ref{cor:sizegenM}, one can then express the module as the row space of a $2(\ell+1)\mu \times (\ell+1)\mu$ matrix $\mathbf{M}^{(0)}$ with entries from $\Fq[x]$. Each element from the row space of this matrix corresponds to an element of $\mathcal M_{s,\ell}(D,G,0)^{(0)}$. The degrees of the polynomials occurring as entries of $\mathbf{M}^{(0)}$, are directly related to the bounds on value $\delta_{G,0}$ takes on the $\ell+1$ generators of $\mathcal M_{s,\ell}(D,G,0)^{(0)}$ stated in Corollary \ref{cor:sizegenM}. Roughly speaking one simply divides the bound from Corollary \ref{cor:sizegenM} by $\mu$, but for the details see \cite{BRS}.

The most costly parts of the decoding algorithm from \cite{BRS} then are as follows:
\begin{enumerate}
\item The computation of the matrix $\mathbf{M}^{(0)}$. This involves among others choosing particularly nice $\Fq[x]$-bases for $\Ya$ and the the $\Ya$-modules $\Ya(G_j)$ and algorithms to compute products of elements from $\Ya$ and $\Ya(G_j)$. In fact, Algorithm 4 in \cite{BRS} is able to compute the product of the $\mu$ $\Fq[x]$-basis elements of $\Ya$ with a given element $a \in \Ya(B)$ for any divisor $B$, in complexity $\softO(\mu^{\omega-1}(\deg B + g+\delta_B(a)))$. Then in Algorithm 5 from \cite{BRS} the products needed to obtain the matrix $\mathbf{M}^{(0)}$ are computed, giving a total complexity of this step $\softO(s\ell^2\mu^{\omega-1}(n+g))$. See Proposition 5.23 and Remark 5.24 in \cite{BRS}.
\item Finding a suitable normal form of $\mathbf{M}^{(0)}$ with the property that one of its rows corresponds to a nonzero element of $\mathcal M_{s,\ell}(D,G,0)^{(0)}$ with lowest possible $\delta_{G,0}$ value. This element then gives rise to the polynomial $Q(z)$ used in the Guruswami-Sudan list-decoder. Here the main ingredient is the computation of the mentioned suitable normal form of $\mathbf{M}^{(0)}$, which turns out to be a certain Popov form of $\mathbf{M}^{(0)}$ with respect to suitably chosen weights. The complexity of this step depends on the size of the matrix (more precisely it gives a factor of the form $(\ell\mu)^\omega$) and the maximum degree among all entries in $\mathbf{M}^{(0)}$, which is of the order of magnitude $s(n+g)/\mu$, essentially using the bounds in Corollary \ref{cor:sizegenM} divides by $\mu$. This gives a total complexity $\softO((\ell\mu)^{\omega}s(n+g)/\mu)=\softO(s\ell^{\omega}\mu^{\omega-1}(n+g))$. See Remark 5.26 and the discussion directly preceding it in \cite{BRS}.
\end{enumerate}
When considering $e>0$, exactly the same approach can be used. Indeed, from a computational points of view the difference is that the divisors $G_j$ are replaced by the divisors $p^eG_j$ and that the bounds in Corollary \ref{cor:sizegenM} have increased at most with a factor $p^e$. On the other hand, the number of generators of the module $\mathcal M_{s,\ell}(D,G,0)^{(e)}$, as indicated in Corollary \ref{cor:sizegenM} did \emph{not} increase: just as in \cite{BRS} we have $2(\ell+1)$ generators. Moreover, the expressions of the form $(z-R)^{p^ej}$ are of course equal to $(z^{p^e}-R^{p^e})^{j}$ for all $j$ according to the Freshman's dream.  As a result, simply following the approach in \cite{BRS} step by step, the complexity of calculating the desired $Q(z)$ as in Theorem \ref{thm:GS_new} is just increased with a factor $p^e$. Hence finding $Q(z)$ is possible in complexity $\softO(s \ell^\omega \mu^{\omega-1}p^e(n+g))$.

The next step in the decoding algorithm from \cite{BRS} is the root finding step. In our case, we know that $Q(z)=\sum_{j=0}^\ell Q_j z^{p^ej}$ for suitable $Q_j \in L(-jp^eG)$. Let us first review the case $e=0$ as managed in \cite{BRS}. What is done there is to find approximations of all roots of $Q(z)$ as power series in $x$ first. More precisely, it is shown in \cite{BRS} that finding all roots of $Q(z)$ in power series form up to precision $\beta$ can be done in complexity $\softO(\mu \beta \deg_z Q)$. One should choose $\beta \in \bigO(\ell(n+g))$ in order to make sure that the sought roots in $L(G)$ can be reconstructed from their power series approximations uniquely. This choice of $\beta$ comes from the fact that the degrees of the divisors $G_j$ all are in $\bigO(\ell(n+g))$. In fact, as observed in \cite{BN}, one can use equation \eqref{eq:G_nslinv} to show that $\beta \in \bigO(s(n+g)).$ Converting a power series approximation of precision $\beta$ to the corresponding element of $L(G)$ can be done in complexity $\softO(\mu^{\omega-1}\beta)$. For more details see Section V.F in \cite{BRS}. What happens for $e>0$ is first of all that the precision needs to be increased to $\beta^{(e)}:=p^e \cdot \beta$, which is $\bigO(p^e \ell(n+g))$. Further, rather than finding approximate power series roots of $Q(z)$, it is faster to find approximate power series roots of the polynomial $\overline{Q}(w)=\sum_{j=0}^\ell Q_j w^{j}$. Then using exactly the same algorithms as in \cite{BRS}, we can find all power series roots of $\overline{Q}(w)$ up to precision $\beta^{(e)}$ in complexity $\softO(\mu \beta^{(e)} \deg_z \overline{Q})=\softO(\mu p^e \ell(n+g))$. Since we are actually interested in roots of $Q(z)$, we now take the $p^e$-th roots of the found power series approximations. This is possible precisely if these roots are actually power series in $x^{p^ej}$, but since we know that any root of $Q(z)$ has multiplicity a multiple of $p^e$, these are the only power series roots of $\overline{Q}(w)$ that we are interested in. Taking the $p^e$-th power of one such approximate power series therefore takes at most $\bigO(\beta^{(e)}/p^e)$ operations. Since there are at most $\ell$ roots in total, we see that we need $\bigO(\ell \beta)$ many operations for taking all $p^e$-th roots. The result is up to $\ell$ many power series up to precision $\beta$ that all are roots with multiplicity a multiple of $p^e$ of $Q(z)$. Finally converting these powers to elements of $L(G)$ costs at most $\softO(\ell \mu^{\omega-1}\beta)$ operations. Combining all the above, we see that the root-finding step can be performed in complexity $\softO(s\ell \mu^{\omega-1}p^e(n+g))$. Hence the complete list-decoding algorithm of the code $C_L(D,G)$ using the approach from Theorem \ref{thm:GS_new} is possible in complexity $\softO(s \ell^\omega \mu^{\omega-1}p^e(n+g))$. We formulate this as a theorem.

\begin{theorem}
Suppose $A=0$. Computing the polynomial $Q(z)$ with the properties as given in Theorem \ref{thm:GS_new} as well as its roots in $L(G)$, can be done in complexity $\softO(s\ell^{\omega} \mu^{\omega-1}p^e(n+g)).$
\end{theorem}

Choosing $e$ small will already significantly reduce the genus penalty, as shown in Corollary \ref{cor:insepGS_t1}, while choosing $e=\lceil \log_p(1+g\ell)\rceil$, will remove it entirely. We can now make the final comments in Section \ref{sec:Freshman} more precise. Since $p^{\lceil \log_p(1+g\ell)\rceil} < p (1+g\ell)$,  list-decoding without genus penalty, can be performed in $\softO(s \ell^{\omega} \mu^{\omega-1}p (1+g\ell)(n+g))$ operations, which simplifies to $\softO(s \ell^{\omega} \mu^{\omega-1}(1+g\ell)(n+g))$ if we consider the characteristic as a constant. In particular, bounded distance decoding up to half the designed minimum distance $d^*=n-\deg G$ can be performed in $\softO(\mu^{\omega-1} (1+g)(n+g))$ operations if the characteristic $p$ is viewed as a constant. We consider several examples, starting with one-point Hermitian codes.

\begin{example}\label{ex:Hermitian_one_point}
Consider a one-point code on the Hermitian function field $F=\mathbb{F}_{q^2}(x,y)$, where $y^q+y=x^{q+1}$. This function field has genus $q(q-1)/2$. Further, let $P_\infty$ be the unique pole of $x$ and $y$. Then the Weierstrass semigroup of $P_\infty$ is well known to be generated by $q$ and $q+1$, so that $\mu=q$. Choose $G=m P_\infty$ and $D$ the sum of the remaining $q^3$ rational places, so that in particular $n=q^3$. Note that
\[\lceil \log_p(1+g)\rceil \le \lceil \log_p(q^2)\rceil=\log_p(q^2).\]
Then the AG code $C_L(D,G)$ can be decoded up to half its designed minimum distance in complexity $\softO(\mu^{\omega-1}p^{\log_p(q^2)}(n+g))=\softO(q^{\omega+4})$ with our algorithm. Using the vectorial Berlekamp-Massey-Sakata algorithm, the same codes can be decoded up to their designed minimum distance in complexity $\bigO(\mu n^2)=\bigO(n^{7/3})=\bigO(q^7)$, see \cite{handbook}. The algorithm from \cite{sakata_2018} has the same complexity, while the algorithm from \cite{LBO} has complexity $\bigO(q^8).$ Since $\omega<2.37286$, our decoder decoder will be faster asymptotically. To be fair, one should mention that in \cite{handbook,sakata_2018,LBO} the codes are actually decoded up to the Kirfel-Pellikaan bound on the minimum distance, which may exceed the designed minimum distance $n-\deg G$ for high rates.
\end{example}

\begin{example}\label{ex:Suzuki_one_point}
Let $a \ge 1$ be an integer and define $q_0=2^a$, $q=2q_0^2$. Consider the Suzuki function field $F=\mathbb{F}_{q}(x,y)$, where $y^q+y=x^{q_0}(x^{q}+x)$. The Suzuki function field has genus $g=q_0(q-1)$ and $q^2+1$ rational places. We denote the unique pole of $x$ by $P_\infty$. It is known that the Weierstrass semigroup of $P_\infty$ is generated by $q$, $q+q_0$, $q+2q_0$, and $q+2q_0+1$.
Hence, smallest positive element of the semigroup of $P_\infty$, is $q$, see for example \cite{HS} so that $\mu=q$ in this setting.

Choose $G=m P_\infty$ and $D$ the sum of the remaining $q^2$ rational places, so that in particular $n=q^2$. Note that
\[\lceil \log_2(1+g)\rceil \le \lceil \log_2(q_0 q)\rceil=3a+1.\]
Then the AG code $C_L(D,G)$ can be decoded up to its designed minimum distance in complexity $\softO(q^{\omega-1}q_0 q q^2)=\softO(q^{\omega+2.5})$ using our decoding algorithm. The given complexity of the decoding algorithm from \cite{LBO} (resp. \cite{handbook}) is in this case $\bigO(q^{5.5})$ (resp. $\bigO(q^{6})$). The complexity estimate of our algorithm compares favourably to this, since $\omega<3$. If no fast linear algebra is used, i.e., if we set $\omega=3$, we obtain the same complexity.

Using \cite{sakata_2018}, these codes can be decoded up to their designed minimum distance in complexity $\bigO(\mu n^2)=\bigO(q^{5})$, see \cite{sakata_2018}. Since $\omega<2.37286$, our decoder will be faster asymptotically. However, unlike in the case of one-point Hermitian codes, using classical algorithms in linear algebra (i.e. setting $\omega=3$), would give a worse complexity.
\end{example}

\begin{example}\label{ex:Ree_one_point}
Let $a \ge 1$ be an integer and define $q_0=3^a$, $q=3q_0^2$. Consider the Ree function field $F=\mathbb{F}_{q}(x,y_1,y_2)$, where $y_1^q-y_1=x^{q_0}(x^{q}-x)$ and $y_2^q-y_2=x^{2q_0}(x^{q}-x)$. Similar to the Suzuki function field, it is known that this number of rational places is maximal for a function field of genus $3q_0(q-1)(q+q_0+1)/2$ over $\Fq$. We denote the unique pole of $x$ by $P_\infty$. For this and more information on the Suzuki function field, see Section 5.4 from \cite{Serre2020}. The smallest positive element of the semigroup of $P_\infty$, is known to be $q^2$, see for example Lemma 5.4 in \cite{DE}. Hence $\mu=q^2$ in this setting. The Weierstrass semigroup at $P_\infty$ is not known, except for $s=1$, but has strictly more than $13$ generators, see \cite{DE}. In fact, for $s=1$, its minimal set of generators has cardinality $132$. 

Now choose $G=m P_\infty$ and $D$ the sum of the remaining $q^3$ rational places, so that in particular $n=q^3$. Note that
\[\lceil \log_3(1+g)\rceil \le \lceil \log_3(3q_0 q^2)\rceil=5a+3.\]
Then the AG code $C_L(D,G)$ can be decoded up to half its designed minimum distance in complexity $\softO((q^2)^{\omega-1}q_0 q^2 q^3)=\softO(q^{2\omega+3.5})$ using our decoding algorithm.

Since the Weierstrass semigroup of $P_\infty$ is so large, the complexity from \cite{handbook} will be far worse than ours. The complexity from \cite{LBO} is $\bigO(q_0q^2(q^3)^2)=\bigO(q^{8.5})$. Our results compare favourably with \cite{LBO}, if we use fast linear algebra, but using classical algorithms ($\omega=3$), we would be worse off. Finally, using \cite{sakata_2018}, the algorithm with a genus penalty, the same codes can be decoded up to their designed minimum distance in complexity $\bigO(\mu n^2)=\bigO(q^{8})$. With the current state of the art concerning the value of $\omega$, our algorithm will therefore have a larger complexity.
\end{example}

\begin{remark}
Using algorithms from \cite{BN} the assumption $A=0$ can be relaxed. In order to compute $\mathcal M_{s,\ell}(D,G,0)^{(e)}$, we used Algorithm 4 in \cite{BRS}, which is able to compute the product of the elements in $\Ya$ with a given element of $\Ya(B)$ for any divisor $B$. This time, since $A \neq 0$, we need a generalization of this algorithm from \cite{BN}. Algorithm 1 from \cite{BN} is able to compute the product of $\mu$ basis elements of $\Ya(B')$ with a given element $a \in \Ya(B)$ for any divisors $B',B$, with complexity $\softO(\mu^{\omega-1}(\deg B + g+\delta_B(a)))$. Otherwise, one can proceed exactly as in the case $A=0$. As long as $\mathfrak a(A)$ is in $\bigO(\deg A)$, the resulting complexity is the same as for $A=0$.
\end{remark}

\begin{remark}
The basic algorithm from \cite{SV}, as well as the GS-list-decoding algorithm with $s=\ell=1$, actually often does correct up to the designed minimum distance. It was shown in \cite{JNH,HL} that the basic algorithm fails with probability around $1/q$ if the code $C_L(D,G)$ is defined over $\Fq$ and $q$ is large. Hence one could decode $C_L(D,G)$ as follows: first run the basic algorithm (or equivalently the GS list-decoder with $\ell=s=1$). In case of decoding failure, run our decoder with $s=\ell=1$ and $e$ such that $p^e \approx 1+g$. Then the \emph{average} complexity is $\softO(\mu^{\omega-1}(n+g)+q^{-1}\mu^{\omega-1}(1+g)(n+g))\softO(\mu^{\omega-1}(1+g/q)(n+g)).$ In case of one-point Hermitian codes over $\mathbb{F}_{q^2}$ for large $q$, this will result in an average complexity $\softO(\mu^{\omega-1}(n+g)=\softO(q^{\omega+2})$. Hence in the average complexity, the effect of running the inseparable GS decoder is not even visible.
\end{remark}

\section{Conclusion}

In this paper we have given a variation of the GS list-decoding algorithm for AG codes using the Freshman's dream. Our variation can be used to completely remove the genus penalty in the list-decoding radius of the usual GS list-decoder. We showed that the complete list-decoding algorithm of the code $C_L(D,G)$ using the approach from Theorem \ref{thm:GS_new} is possible in complexity $\softO(s \ell^\omega \mu^{\omega-1}p^e(n+g))$. Choosing $e$ small will already significantly reduce the genus penalty, as shown in Corollary \ref{cor:insepGS_t1}, while choosing $e=\lceil \log_p(1+g\ell)\rceil$, will remove it entirely. The inseparable Guruswami-Sudan list-decoder without genus penalty can be performed in $\softO(s \ell^{\omega} \mu^{\omega-1}(1+g\ell)(n+g))$ operations if the characteristic is constant. In particular, bounded distance decoding up to the designed minimum distance $d^*=n-\deg G$ can be performed in $\softO(\mu^{\omega-1}(1+g)(n+g))$ operations.

\end{document}